\documentclass{birkjour}
\usepackage{longtable}
\usepackage{booktabs}
\usepackage{multirow}
\usepackage[T2A,OT1]{fontenc}
\usepackage[utf8]{inputenc}
\usepackage[russian,english]{babel}

 \newtheorem{thm}{Theorem}[section]
 \newtheorem{cor}[thm]{Corollary}
 \newtheorem{lem}[thm]{Lemma}

 \theoremstyle{definition}
 \newtheorem*{defn}{Definition}
 
 \theoremstyle{remark}
 \newtheorem*{rem}{Remark}
 
 \numberwithin{equation}{section}
 
 \newcommand{\myfnmark}[1]{\mbox{\textsuperscript{\normalfont #1}}}

\begin{document}

\title{Sums and Products of Regular Polytopes' Squared Chord Lengths}

\author{Jessica N. Copher}
\address{Department of Mathematics\br Box 8205\br North Carolina State University\br Raleigh, NC  27695-8205}
\email{jncopher@ncsu.edu}

\begin{abstract}
Although previous research has found several facts concerning chord lengths of regular polytopes, none of these investigations has
considered whether any of these facts define relationships that might
generalize to the chord lengths of \emph{all} regular polytopes. Consequently,
this paper explores whether four findings of previous studies\textemdash viz., the
four facts relating to the sums and products of squared chord
lengths of regular polygons inscribed in unit circles\textemdash can
be generalized to all regular ($n$-dimensional) polytopes (inscribed
in unit $n$-spheres). We show that (a) one of these four facts actually does generalize to all regular polytopes (of dimension $n\geq 2$), (b) one generalizes to all regular polytopes
except most simplices, (c) one generalizes only to the
family of crosspolytopes, and (d) one generalizes only to the crosspolytopes
and 24-cell. We also discover several corollaries (due to reciprocation)
and some theorems specific to the three-dimensional regular polytopes
along the way.
\end{abstract}



\maketitle

\section{Introduction}
Several discoveries have been made concerning the chord lengths of
certain types of regular polytopes (most often, the regular two-dimensional polytopes, i.e., regular polygons; e.g., see \cite{FontaineHurley,Kappraff2002,SasaneChapman,Steinbach}).
However, no single rule has been found to apply to the chord lengths
of \emph{all} regular polytopes.

To remedy this situation, we consider whether any of the four facts
discovered in one area of previous investigation (regarding the chord
lengths of two-dimensional regular polytopes) can be generalized to
apply to the chord lengths of \emph{all} $n$-dimensional regular 
polytopes where $n\geq 2$.\footnote{Although the 0-dimensional polytope (i.e., a single point) and the 1-dimensional polytope (i.e., a line segment) are regular, we exclude them as being trivial.}  (In this paper, we will call an $n$-dimensional
polytope an ``$n$-polytope,'' an $n$-dimensional sphere an
``$n$-sphere,'' etc.).
\subsection{The Four Known Facts}

First, recall that a \emph{chord} of a regular 2-polytope (i.e., regular polygon) is a line segment whose two endpoints are vertices of the
polygon.

Let $\mathcal{P}$ be a regular polygon with $E$ edges that is inscribed
in a unit circle (i.e., ``2-sphere'').
\begin{enumerate}
\item The sum of the squared chord lengths of $\mathcal{P}$ equals $E^{2}$.
\item The sum of the squared \emph{distinct} chord lengths of $\mathcal{P}$
equals:
\begin{enumerate}
\item $E$ (when $E$ is odd) \cite{MorleyHarding}
\item some integer (when $E$ is even) \cite[pp. 490-491]{Kappraff2002}.
\end{enumerate}
\item The product of the squared chord lengths of $\mathcal{P}$ equals
$E^{E}$.
\item The product of the squared \emph{distinct }chord lengths of $\mathcal{P}$
equals:
\begin{enumerate}
\item $E$ (when $E$ is odd)
\item some integer (when $E$ is even) \cite[pp. 490-491]{Kappraff2002}.
\end{enumerate}
\end{enumerate}
(Facts 1 and 3 are given in an unpublished 2013 article\footnote{Mustonen's unpublished article is titled
``Lengths of edges and diagonals and sums of them in regular polygons as roots of algebraic equations'' and can be accessed at http://www.survo.fi/papers/Roots2013.pdf as of 16 March 2019.

Although Mustonen states Facts 1 and 3 without mathematical proof,
a proof of Fact 1 may be obtained by substituting $R=1$ in Solution 6.73 of \cite{Prasolov}. (An English version of \cite{Prasolov}, edited and translated by Dimitry Leites, is in preparation with the title, \emph{Encyclopedia of Problems in Plane and Solid Geometry}).

Likewise, Fact 3 can be derived from a statement proven in \cite[pp. 161-162]{Honsberger}, namely, that if $\mathcal{P}$
is a regular polygon with $E$ edges inscribed in a unit circle,
the product $p$ of the lengths of the chords of $\mathcal{P}$ emanating from a given vertex $\mathbf{P}$ equals $E$. Denoting the vertices of $\mathcal{P}$ by $\mathbf{P_i}$ for $i=1,2,...,V$ and their associated products by $p_i$, we have  $p_1p_2...p_V=p^V=E^{V}$. Since each chord's length occurs twice in $p_1p_2...p_V$, this is also equal to the product of the \emph{squared} chord lengths
of $\mathcal{P}$. Since, for regular polygons, $V=E$ (see \cite{CoxeterGreitzer}),
we obtain Fact 3.} by S. Mustonen).

\subsection{Overview of Procedure}

To test the generalization of these four 2-polytope facts to all dimensions (greater than or equal to 2),
we apply the following four steps:
\begin{enumerate}
\item Consider a regular $n$-polytope inscribed in a
unit $n$-sphere ($n \geq 2$).
\item Compute both the sum and product of (a) the polytope\textquoteright s
squared chord lengths and (b) the polytope\textquoteright s squared
distinct chord lengths.
\item Compare the results obtained in Step 2 with the values (which are
usually some variant on $E$) given by the relevant 2-polytope fact.
\item If the compared values in Step 3 are different, compare the results
obtained in Step 2 with other values related to the $n$-polytope
(e.g., other $j$-face cardinalities).
\end{enumerate}

The next section contains a review of some basic
definitions and properties necessary to carrying out these steps.
In the four remaining sections, we carry out these steps
for each of the four facts in turn.

\section{Preliminaries}

In this paper, all $n$-polytopes being considered are \emph{finite
convex} regions of Euclidean $n$-space.  For a full definition, see \cite[pp.126-127]{Coxeter1973}.

An $n$-polytope $\mathcal{P}$ has one or more \emph{$j$-faces}
where $j \in \left \{-1,0,1,2,...,n \right \}$ (cf. \cite{Matousek}). For $j=0,1,...,n-1$,
each of these $j$-faces is the $j$-dimensional intersection of $\mathcal{P}$
and $\mathcal{H}$ (where $\mathcal{H}$ is an $\left(n-1\right)$-plane such that $\mathcal{P}$ lies in one of the closed half-spaces
determined by $\mathcal{H}$ and has a nonempty intersection with
$\mathcal{H}$) \cite{Matousek}. For $j=n$, the
sole $j$-face is just the $n$-polytope $\mathcal{P}$ itself \cite{Matousek},
and, for $j=-1$, the sole $j$-face is the empty set \cite{Matousek}.
For any $n$-polytope, the 0-faces are called \emph{vertices},
the 1-faces are called \emph{edges}, the $(n-2)$-faces are called \emph{ridges}, and the $(n-1)$-faces are called \emph{facets} \cite{Matousek}.

\sloppy An $n$-polytope is called \emph{regular} if:
\begin{enumerate}
\item Whenever one of its $\left(j-2\right)$-faces is incident with one
of its $j$-faces, there are exactly two\emph{ }$\left(j-1\right)$-faces
that are incident with both the $\left(j-2\right)$-face and the $j$-face
(for $j=1,2,...,n$), and
\item For a given $j\in\left\{ 0,1,2,...,n\right\} $, all of the $j$-faces
are equidistant from a point $\mathbf{O}$ called the polytope's \emph{center}
\cite{Coxeter1965}. (That is, for each $j\in\left\{0,1,2,..., n\right\} $,
there is an $n$-sphere centered at $\mathbf{O}$ that passes through
all the $j$-faces' centers). Without loss of generality, all polytopes
will be considered as being centered at the origin.
\end{enumerate}

One useful consequence of regularity is that, for a given $j$, any
of a polytope\textquoteright s $j$-faces can be interchanged with
any of the polytope's other $j$-faces via one of its \emph{symmetries}
\cite{Grunbaum}. A second useful consequence of regularity
is that, when a regular $n$-polytope is circumscribed
about an $n$-sphere, $n\geq2$, it may be reciprocated with respect to the $n$-sphere
to form a new regular $n$-polytope inscribed inside the $n$-sphere
(this maps the original polytope's $\left(j-1\right)$-faces to the
new polytope's $\left(n-j\right)$-faces for $j=0,1,...,n$) \cite{Coxeter1973}.

There are exactly five regular 3-polytopes, six
regular 4-polytopes, and three regular $n$-polytopes in each dimension
$n\geq5$ \cite{Coxeter1973}. Table \ref{tab:j-Face-Cardinality-and-Shape} shows the number and shape of the $j$-faces
of each of these polytopes.

\begingroup
\makeatletter
\renewcommand\@makefntext[1]%
   {\noindent\makebox[1.8em][r]{\myfnmark}#1} 
\makeatother
\renewcommand{\thefootnote}{\alph{footnote}}

\begin{longtable}{ccccccc}
\caption{\label{tab:j-Face-Cardinality-and-Shape}Cardinality and Shape of
$j$-Faces of Regular Polytopes \cite{Coxeter1973}.}
\tabularnewline
\hline
\midrule 
\textbf{Polytope} & \multicolumn{4}{c}{\textbf{$j$-Face Cardinality}} & \multicolumn{2}{c}{\textbf{Shape of $j$-Face}\footnotemark[1]}\tabularnewline
\endfirsthead
\tabularnewline
\midrule 
\textbf{3-polytopes} & \textbf{0-face} & \textbf{1-faces} & \textbf{2-face} &  & \textbf{2-face} & \tabularnewline
\midrule 
\emph{tetrahedron} & 4 & 6 & 4 &  & \small triangle & \tabularnewline
\midrule 
\emph{octahedron} & 6 & 12 & 8 &  & \small triangle & \tabularnewline
\midrule 
\emph{cube} & 8 & 12 & 6 &  & \small square & \tabularnewline
\midrule 
\emph{icosahedron} & 12 & 30 & 20 &  & \small triangle & \tabularnewline
\midrule 
\emph{dodecahedron} & 20 & 30 & 12 &  & \small pentagon & \tabularnewline
\midrule 
\textbf{4-polytopes} & \textbf{0-face} & \textbf{1-face} & \textbf{2-face} & \textbf{3-face} & \textbf{2-face} & \textbf{3-face}\tabularnewline
\midrule 
\emph{5-cell} & 5 & 10 & 10 & 5 & \small triangle & \small tetrahedron\tabularnewline
\midrule 
\emph{16-cell} & 8 & 24 & 32 & 16 & \small triangle & \small tetrahedron\tabularnewline
\midrule 
\emph{8-cell} & 16 & 32 & 24 & 8 & \small square & \small cube\tabularnewline
\midrule 
\emph{24-cell} & 24 & 96 & 96 & 24 & \small triangle & \small octahedron\tabularnewline
\midrule 
\emph{600-cell} & 120 & 720 & 1200 & 600 & \small triangle & \small tetrahedron\tabularnewline
\midrule 
\emph{120-cell} & 600 & 1200 & 720 & 120 & \small pentagon & \small dodecahedron\tabularnewline
\midrule 
\textbf{$n$-polytopes} & \multicolumn{4}{c}{\textbf{any $j$-face}} & \multicolumn{2}{c}{\textbf{any $j$-face}}\tabularnewline
\midrule 
\emph{$n$-simplex} & \multicolumn{4}{c}{$\binom{n+1}{j+1}$} & \multicolumn{2}{c}{\small $j$-simplex}\tabularnewline
\midrule 
\parbox{2cm}{\centering\emph{$n$-crosspolytope}} & \multicolumn{4}{c}{$2^{j+1}\binom{n}{j+1}$} & \multicolumn{2}{c}{\small $j$-simplex}\tabularnewline
\midrule 
\emph{$n$-cube} & \multicolumn{4}{c}{$2^{n-j}\binom{n}{j}$} & \multicolumn{2}{c}{\small $j$-cube}\tabularnewline
\bottomrule
\end{longtable}

\begin{center}
\begin{minipage}{.95\linewidth}
\renewcommand{\footnoterule}{}
\footnotetext[1]{\myfnmark{a} The shape of a 0-face is a point; the shape of a 1-face is a line segment.}
\end{minipage}
\end{center}
\endgroup
\bigskip
\noindent Lastly, a \emph{chord} of a regular $n$-polytope is being defined,
in this paper, as a line segment whose endpoints are vertices (i.e.,
0-faces) of the polytope. Thus, the chords of a regular polytope consist
of all of its edges and diagonals.

\section{FACT 1: Sum of Squared Chords}

The first 2-polytope fact stated that, for any regular 2-polytope
inscribed in a unit 2-sphere, the sum of its squared chord lengths
equals $E^{2}.$ To test generalization of this fact to all regular
$n$-polytopes ($n\geq 2$), we will make use of the following three lemmas and definition.

\subsection{\label{Sect:3.1}Preliminaries}
\begin{lem}
\label{Lemma:edge-length}Let $\mathcal{P}$ be a regular $n$-polytope
inscribed in a unit $n$-sphere, $n\geq2$, and let $x$ denote
the ratio of the $n$-sphere's radius $_{0}R$ to the polytope's half-edge
length $l$. Then the edge length $e$ of $\mathcal{P}$ is given by the
formula $e=\frac{2}{x}$.
\end{lem}
\begin{proof}
From the hypotheses, we have $x=\frac{_{0}R}{l}$. To find the edge
length $e$ (which is equal to $2l$), we first solve for $2l$ and
then replace $_{0}R$ with $1$ (since the $n$-sphere has a radius
of 1). We obtain: $e=2l=\frac{2_{0}R}{x}=\frac{2}{x}$.
\end{proof}
\begin{rem}
To use Lemma \ref{Lemma:edge-length} in our proofs, we will substitute
values (or formulas) for $x$ obtained from \cite{Coxeter1973} for
select polytopes. For the regular 3- and 4-polytopes, we will substitute
the values of $x$ that are listed in the tables of ``regular polyhedra
in ordinary space'' and ``regular polytopes in four dimensions''
in \cite{Coxeter1973} (i.e., Table Ii-ii on pp. 292-293)\footnote{The values of $x$ are referred to as the values of $\frac{_{0}R}{l}$
in \cite{Coxeter1973}.} and are also given in the cells marked with an asterisk in our Table
\ref{tab:Edge-Length}. For the regular $n$-simplices, $n$-crosspolytopes,
and $n$-hypercubes, we will first substitute $j=0$ into the formulas
for $\frac{_{j}R}{l}$ (the ratio of the \textquotedblleft radius
of $n$-sphere intersecting each $j$-face\textquoteright s center\textquotedblright{}
to \textquotedblleft half-edge length\textquotedblright ) listed in
the table of ``regular polytopes in $n$ dimensions'' in \cite{Coxeter1973}
(i.e., Table Iiii on pp. 294-295) to obtain formulas for $\frac{_{0}R}{l}$
(i.e., formulas for $x$) that we can then substitute into Lemma \ref{Lemma:edge-length}
to obtain the edge length $e$ for each of these three $n$-polytopes.
Each of these steps are shown in the last three rows on Table \ref{tab:Edge-Length}.

\end{rem}

\begingroup
\makeatletter
\renewcommand\@makefntext[1]%
   {\noindent\makebox[1.8em][r]{\myfnmark}#1} 
\makeatother
\renewcommand{\thefootnote}{\alph{footnote}}

\begin{longtable}{ccccc}
\caption{\label{tab:Edge-Length}Edge Length, $e$, of Select Regular Polytopes.}
\tabularnewline
\hline
\midrule 
 & \textbf{Polytope} & $\frac{_{j}R}{l}$ & $\frac{_{0}R}{l}$ ( $=x$) & $e$\tabularnewline
\midrule
\endfirsthead
{\multirow{3}{*}{\textbf{3-polytopes}\footnotemark[1]}} & \emph{cube} & NA & $\sqrt{3}${*} & $\frac{2}{\sqrt{3}}$\tabularnewline
\cmidrule{2-5} 
 & \emph{icosahedron} & NA & $\sqrt{\tau\sqrt{5}}${*} & $\frac{2}{\sqrt{\tau\sqrt{5}}}$\tabularnewline
\cmidrule{2-5} 
 & \emph{dodecahedron} & NA & $\tau\sqrt{3}${*} & $\frac{2}{\tau\sqrt{3}}$\tabularnewline
\midrule 
\textbf{4-polytope} & \emph{24-cell} & NA & $2${*} & $1$\tabularnewline
\midrule 
\multirow{3}{*}{\textbf{$n$-polytopes}\footnotemark[2]} & \emph{$n$-simplex} & $\sqrt{\frac{2}{j+1}-\frac{2}{n+1}}$ & $\sqrt{2-\frac{2}{n+1}}$ & $\frac{2}{\sqrt{2-\frac{2}{n+1}}}$\tabularnewline
\cmidrule{2-5} 
 & \emph{$n$-crosspolytope} & $\sqrt{\frac{2}{j+1}}$ & $\sqrt{2}$ & $\sqrt{2}$\tabularnewline
\cmidrule{2-5} 
 & \emph{$n$-cube} & $\sqrt{n-j}$ & $\sqrt{n}$ & $\frac{2}{\sqrt{n}}$\tabularnewline
\bottomrule
\end{longtable}

\begin{center}
\begin{minipage}{.95\linewidth}
\renewcommand{\footnoterule}{}
\footnotetext[1]{\myfnmark{a} In these rows (and throughout the paper), $\tau$ denotes the golden ratio $\frac{1+\sqrt{5}}{2}$.}
\footnotetext[2]{\myfnmark{b} Although Coxeter \cite{Coxeter1973} only states that his formulas
for $\frac{_{j}R}{l}$ are valid for $n\geq5$ (see his Table Iiii),
it is evident from his development of these formulas (see pp. 133-134,
158-159) that they are, in fact, valid for
$n\geq2$.}
\end{minipage}
\end{center}
\endgroup
\bigskip
\begin{lem}
\label{Lemma:mV/2}Let $\mathcal{P}$ be a regular polytope with $V$
vertices, and let $m_{i}$ be the number of chords of length $d_{i}$
emanating from a given vertex of $\mathcal{P}$. Then the total number
of chords of length $d_{i}$ is given by $N_{i}=\frac{m_{i}V}{2}$.
\end{lem}
\begin{proof}
Let $\mathbf{P_{1},P_{2},\ldots,P_{V}}$ be the vertices of $\mathcal{P}$.
Without loss of generality, consider the vertex $\mathbf{P_{1}}$.
If vertex $\mathbf{P_{1}}$ has $m_i$ chords of length $d_{i}$ emanating
from it, then, by the mappings induced by regularity, each of the
$V$ vertices of the polytope also has $m_i$ chords of length $d_{i}$
emanating from it. However, since the chord $\mathbf{P_{j}P_{j'}}$
is equal to the chord $\mathbf{P_{j'}P_{j}}$ for all distinct $j,j'\in\left\{ 1,2,\ldots,V\right\} $,
this yields a total of only $\frac{m_iV}{2}$ chords of length $d_{i}$. 
\end{proof}
\begin{lem}
\label{lem:N=00003DV(V-1)/2}Let $\mathcal{P}$ be a regular polytope
with $V$ vertices and $N$ total chords. Then $N=\frac{V\left(V-1\right)}{2}$.
\end{lem}
\begin{proof}
Since each chord in $\mathcal{P}$ corresponds to a single pair of
distinct vertices, the total number of chords $N$ is equal to the
total number of pairs of distinct vertices, i.e., $N=\binom{V}{2}=\frac{V\left(V-1\right)}{2}$.
\end{proof}
\begin{defn}
An $n$-polytope is called \emph{centrally symmetric} if there exists
a symmetry of the polytope that transforms each point $\left(x_{1},\;x_{2},\;\ldots\,,\;x_{n}\right)$
of the polytope into the point $\left(-x_{1},\;-x_{2},\;\ldots\,,\;-x_{n}\right)$
(cf. \cite{Coxeter1973,Grunbaum}).
\end{defn}
\noindent A useful property of centrally-symmetric polytopes is that, for each
vertex $\mathbf{P}$, there is a unique vertex $\mathbf{P}'$ lying
diametrically directly opposite to it (cf. \cite{Coxeter1973}). All regular $n$-polytopes, $n\geq2$, are centrally
symmetric except for the odd-edged polygons and simplices \cite{Coxeter1973}.
\subsection{\label{subsec:Generalizing-Fact-1}Generalizing Fact 1}
\begin{thm}
\label{Thm:sum-of-all-chords-Vsquared}Let $\mathcal{P}$ be any regular
$n$-polytope inscribed in a unit $n$-sphere (where $n \geq 2$). Then
\[\sum_{i=1}^{N}c_{i}^{2}=V^{2}\]
where $c_{i}$ is the length of each $i$\textsuperscript{th} chord
of $\mathcal{P}$, $N$ is the total number of chords, and $V$ is
the total number of vertices.
\end{thm}
\begin{proof} A regular $n$-polytope, where $n \geq 2$, is either centrally symmetric
or \emph{not} centrally symmetric (in the latter case, it is either
an odd-edged polygon or an $n$-simplex). Thus, we consider three cases: (a) $\mathcal{P}$ is centrally symmetric,
(b) $\mathcal{P}$ is an odd-edged polygon, or (c) $\mathcal{P}$ is an $n$-simplex.

\emph{Case 1 ($\mathcal{P}$ is centrally symmetric).} Let $k$ be the number of \emph{distinct} chords lengths of $\mathcal{P}$.  Without loss of generality,
chose a vertex \textbf{$\mathbf{P}_{\mathbf{0}}$} of the polytope.
Let \textbf{$\mathbf{P_{i}}$} for $i=1,2,...,k$ be $k$ other vertices
of the polytope such that $\mathbf{P}_{0}$ and $\mathbf{P_{i}}$
form the endpoints of $k$ chords each having a distinct length $d_{i}$
and $d_{1}<d_{2}<\cdots<d_{k}$. We now find a formula (or value)
for each of these $k$ distinct lengths.

Observe that\textemdash except in the case where both $\mathbf{P}_{0}$
and $\mathbf{P_{i}}$ are collinear with the center $\mathbf{O}$
of the unit $n$-sphere\textemdash each chord $\mathbf{P_0P_i}$ 
is the side of a triangle $\triangle \mathbf{P}_{0} \mathbf{O} \mathbf{P_{i}}$.
Moreover, the other two sides of the triangle ($\mathbf{OP_{0}}$
and $\mathbf{OP_{i}}$) form radii of the $n$-sphere and so have length 1. Applying the Law of Cosines, we have $d_{i}^{2}=1^{2}+1^{2}-2\left(1\right)\left(1\right)\cos\theta_{i}=2-2\cos\theta_{i}$
where $\theta_{i}$ is the angle between the two unit-long sides of
the triangle.

When both $\mathbf{P}_{0}$
and $\mathbf{P_{i}}$ are collinear with $\mathbf{O}$, the chord $\mathbf{P_{0}P_{i}}$ is a diameter of the
unit $n$-sphere and so has length 2. Since a diameter is the \emph{longest}
possible chord of an $n$-sphere, $d_{k}=2$. Furthermore, only one diameter
can emanate from $\mathbf{P_{0}}$, so Lemma \ref{Lemma:mV/2} reveals that $\mathcal{P}$ has a total
of $\frac{1V}{2}$ chords of length $d_{k}$.

Letting $N_{i}$ be the number of chords of length $d_{i}$ ($i=1,2,...,k$), we have
\[
\sum_{i=1}^{N}c_{i}^{2}=\sum_{i=1}^{k}N_{i}d_{i}^{2}=\sum_{i=1}^{k-1}N_{i}\left(2-2\cos\theta_{i}\right)+N_{k}d_{k}^{2}=
\]
\[
2\sum_{i=1}^{k-1}N_{i}-2\sum_{i=1}^{k-1}N_{i}\cos\theta_{i}+\frac{V}{2}\left(2^{2}\right)=
2\left(\sum_{i=1}^{k}N_{i}-\frac{V}{2}\right)-2\sum_{i=1}^{k-1}N_{i}\cos\theta_{i}+2V.
\]
Since $\sum_{i=1}^{k}N_{i}$ must equal the total number of chords
(i.e., $N$) and since $N=\frac{V\left(V-1\right)}{2}$ by Lemma \ref{lem:N=00003DV(V-1)/2},
we have:
\[
\sum_{i=1}^{N}c_{i}^{2}=2\left[\frac{V\left(V-1\right)}{2}-\frac{V}{2}\right]-2\sum_{i=1}^{k-1}N_{i}\cos\theta_{i}+2V=V^{2}-2\sum_{i=1}^{k-1}N_{i}\cos\theta_{i}.
\]
Next, we show that $\sum_{i=1}^{k-1}N_{i}\cos\theta_{i}=0$. Since
$\mathcal{P}$ is centrally symmetric, each vertex $\mathbf{P}_{\mathbf{i}}$
has an opposite vertex $\mathbf{P_{i}'}$ for $i=0,1,...,k$. Thus, for $i\in\left\{1,2,...,k-1\right\}$,
$\mathbf{P_{0}P_{k}}$ and $\mathbf{P_{i}P_{i}'}$
intersect at $\mathbf{O}$ to form vertical angles $\angle\mathbf{P_{0}OP_{i}}$
and $\angle\mathbf{P_{k}OP_{i}'}$ .
Hence, these two angles are congruent. Furthermore, observe that $\angle\mathbf{P_{0}OP_{i}'}$ is the supplement of $\angle\mathbf{P_{0}OP_{i}}$
(which has measure $\theta_{i}$). Thus, $\angle\mathbf{P_{0}OP_{i}'}$
has measure $\pi-\theta_{i}$. However, the line segment \textbf{$\mathbf{P_{0}P_{i}'}$}
(lying opposite to the angle $\angle\mathbf{P_{0}OP_{i}'}$) is a
side of the triangle $\triangle\mathbf{P_{0}OP_{i}'}$ (whose other
sides, $\mathbf{OP_{0}}$ and $\mathbf{OP_{i}'}$, are radii of length
1) and, moreover, must be the same length as one of the original\emph{
$\mathbf{P_{0}P_{i}}$} for $i=1,2,...,k-1$ (because these $k-1$
chords are of all possible distinct chord lengths\textemdash except
for that of the unique diameter emanating from $\mathbf{P_{0}}$\textemdash and
we know $\mathbf{P_{0}P_{i}'}\neq\mathbf{P_{0}P_{k}}$). Hence, by
side-side-side, $\triangle\mathbf{P_{0}OP_{i}'}$ is congruent to
one of the ``original triangles $\triangle\mathbf{P_{0}OP_{i}}$.''
Thus, the ``original $\mathbf{P_{0}P_{i}}$'' (i.e., the one that
is the same length as $\mathbf{P_{0}P_{i}'}$) must lie opposite to
one of the ``original angles $\angle\mathbf{P_{0}OP_{i}}$,'' \emph{which
must have the same measure as $\angle\mathbf{P_{0}OP_{i}'}$ (i.e.,
$\pi-\theta_i$)}. Thus, we have shown that, for each angle of measure
$\theta_{i}$ contained in the set $\left\{ \angle\mathbf{P_{0}OP_{i}}:\:i=1,2,...,k-1\right\} $,
there is an angle (not necessarily distinct from the angle of measure
$\theta_{i}$) that is also contained in this set and has measure
$\pi-\theta_{i}$. Moreover, since we have shown that, for every chord
of length $d_{i}$ (corresponding to an angle of measure $\theta_{i}$)
that emanates from $\mathbf{P_{0}}$, there is a chord of length $d_{i}'$
(corresponding to an angle of measure $\pi-\theta_{i}$) that also
emanates from $\mathbf{P_{0}}$, Lemma \ref{Lemma:mV/2} tells us
that there must be the same number $N_{i}$ of chords of length $d_{i}$
as chords of length $d_{i}'$. Call the latter number $N_{i}'$.

Returning to $\sum_{i=1}^{k-1}N_{i}\cos\theta_{i}$, this shows that,
if the angle $\angle\mathbf{P_{0}OP_{i}}$ of measure $\theta_{i}$
and its supplement (the one contained in the set $\left\{ \angle\mathbf{P_{0}OP_{i}}:\:i=1,2,...,k\right\} $)
are distinct angles, we have $N_{i}\cos\theta_{i}+N_{i}'\cos\left(\pi-\theta_{i}\right)=N_{i}\cos\theta_{i}+N_{i}\cos\left(\pi-\theta_{i}\right)=N_{i}\left(\cos\theta_{i}-\cos\theta_{i}\right)=0$
and these two angles contribute nothing to $\sum_{i=1}^{k-1}N_{i}\cos\theta_{i}$.
If the angle of measure $\theta_{i}$ and its supplement are actually
the same angle, then $\theta_{i}=\frac{\pi}{2}$ and $N_{i}\cos\left(\frac{\pi}{2}\right)=0$.
Thus, this angle also contributes nothing to $\sum_{i=1}^{k-1}N_{i}\cos\theta_{i}$.
Therefore, we see that $\sum_{i=1}^{k-1}N_{i}\cos\theta_{i}=0$ and
$\sum_{i=1}^{N}c_{i}^{2}=V^{2}.$

\emph{Case 2 ($\mathcal{P}$ is an odd-edged polygon).} Since $\mathcal{P}$ is a regular polygon, Fact 1 tells is $\sum_{i=1}^{N}c_{i}^{2}=E^{2}$. Since the number of vertices, $V$,
of an $E$-edged polygon is equal to $E$ \cite{CoxeterGreitzer},
we have: \emph{$\sum_{i=1}^{N}c_{i}^{2}=V^{2}$.}

\emph{Case 3 ($\mathcal{P}$ is an $n$-simplex).} A simplex has no diagonals
(since all of its vertices are connected by
edges; cf. \cite{Coxeter1973}), thus, all of the chords of $\mathcal{P}$ are
edges. An edge is a 1-face, so the number of edges $E$ is obtained
by letting $j=1$ in the cardinality formula for the $j$-faces of $n$-simplices given
in Table \ref{tab:j-Face-Cardinality-and-Shape}:
\[
E=\binom{n+1}{j+1}=\binom{n+1}{2}=\frac{\left(n+1\right)n}{2}.
\]
The edge length $e$ of an $n$-simplex is $2\left(2-\frac{2}{n+1}\right)^{-1/2}$ (see Table \ref{tab:Edge-Length}). We have:
$$
\sum_{i=1}^{N}c_{i}^{2}=Ee^{2}=\frac{\left(n+1\right)n}{2}\left(\frac{4}{2-\frac{2}{n+1}}\right)=\left(n+1\right)^{2}.
$$
Since the vertices are 0-faces, the number of vertices $V$ of $\mathcal{P}$ can be seen to be $n+1$ from Table \ref{tab:j-Face-Cardinality-and-Shape}.
Therefore, $\sum_{i=1}^{N}c_{i}^{2}=(n+1)^{2}=V^{2}$.
\end{proof}
\begin{cor}
\label{DualCor:sum-of-all-chords}Let $\mathcal{Q}$ be any regular
$n$-polytope circumscribed about a unit $n$-sphere (where $n \geq 2$).
Then $\sum_{i=1}^{N}s_{i}^{2}=F^{2}$ where $s_{i}$ is the length
of each $i$\textsuperscript{th} line segment whose endpoints are
centers of facets of $\mathcal{Q}$, $N$ is the number of these line
segments, and $F$ is the number of facets of $\mathcal{Q}$.
\end{cor}
\begin{proof} The centers of the facets of $\mathcal{Q}$
are also the vertices of a reciprocal regular $n$-polytope $\mathcal{P}$
of $\mathcal{Q}$ with respect to the unit $n$-sphere about which $\mathcal{Q}$ is circumscribed (cf. \cite{Coxeter1973}).
Since the vertices of $\mathcal{P}$ are the centers of the facets
$\mathcal{Q}$, the chords of $\mathcal{P}$ are the line segments
whose endpoints are the centers of facets of $\mathcal{Q}$. Hence,
the lengths of these line segments ($s_{i}$ for $i=1,2,\ldots,N$)
are also the lengths of the chords of $\mathcal{P}$. Thus, by Theorem
\ref{Thm:sum-of-all-chords-Vsquared}, we have $\sum_{i=1}^{N}s_{i}^{2}=V^{2}$ where $V$ is the number
of vertices of $\mathcal{P}$. Moreover, since every facet of $\mathcal{Q}$
corresponds to a vertex of $\mathcal{P}$ (and vice versa), the number
of facets $F$ of $\mathcal{Q}$ is the same as the number of vertices
$V$ of $\mathcal{P}$. Therefore, we have $\sum_{i=1}^{N}s_{i}^{2}=F^{2}$
for $\mathcal{Q}$. \end{proof}

\subsection{Discussion}
Theorem \ref{Thm:sum-of-all-chords-Vsquared} shows that Fact 1 does,
in fact, apply to \emph{all} regular $n$-polytopes where $n \geq 2$\textemdash \emph{if
the \textquotedblleft number of edges\textquotedblright{} ($E$) mentioned
in Fact 1 is reinterpreted as \textquotedblleft number of vertices\textquotedblright{}
($V$).} This reinterpretation is valid since, for 2-polytopes, $E=V$.

\section{FACT 2: Sum of Squared Distinct Chords}
The second 2-polytope fact stated that, for any regular 2-polytope
inscribed in a unit 2-sphere, the sum of its squared distinct chord
lengths equals (a) $E$ (when $E$ is odd) and (b) some integer (when
$E$ is even). To test generalization of this fact to all regular
$n$-polytopes ($n \geq 2$), we use the following lemmas.

\subsection{Preliminaries}
\begin{lem}
\label{Lemma-=0000233}Let $\mathcal{P}$ be a regular 2-polytope
with $E$ edges and $k$ chords of distinct lengths.  If $E$ is odd, then $E=2k+1$.
\end{lem}
\begin{proof}
Let $\mathbf{O}$ be the center of $\mathcal{P}$, let $\mathbf{P_{0}}$
be a vertex of $\mathcal{P}$, and let $\mathbf{P_{i}}$ for $i=1,2,...,k$
be $k$ other vertices of $\mathcal{P}$ such that $\mathbf{P}_{0}$
and $\mathbf{P_{i}}$ form the endpoints of $k$ chords each having
a distinct length $d_{i}$. Since $E$ is odd, the line $l$ that
passes through $\mathbf{O}$ and $\mathbf{P_{0}}$ is an axis of symmetry
of $\mathcal{P}$ (cf. \cite{AgricolaFriendrich}). Furthermore,
only one of the vertices of $\mathcal{P}$ lies on $l$ (i.e., $\mathbf{P_{0}};$
cf. \cite{AgricolaFriendrich}). If we reflect $\mathcal{P}$
across the line $l$, we see that, for each $\mathbf{P_{i}}$ where $i \neq 0$,
there is another vertex, call it $\mathbf{P_{i}'}$, that is the same
distance $d_{i}$ from $\mathbf{P_{0}}$ (because it is now in the
same location as $\mathbf{P_{i}}$ had been before reflecting $\mathcal{P}$
and $\mathbf{P_{0}}$ has not moved). Since there
are no other (non-identity) symmetries of $\mathcal{P}$ that hold
$\mathbf{P_{0}}$ constant, this shows that these two vertices, \textbf{$\mathbf{P_{i}}$}
and $\mathbf{P_{i}'}$, are the only two vertices that are a distance
of $d_{i}$ away from $\mathbf{P_{0}}$ for each $i\in\left\{ 1,2,...,k\right\} $.
Thus, $\mathcal{P}$ has $2k$ vertices besides $\mathbf{P_{0}}$ and
the total number of vertices $V$ of $\mathcal{P}$
is $2k+1$. Since $E=V$ for 2-polytopes, we have $E=2k+1$.
\end{proof}
\begin{lem}
\label{lem:odd-edged-non-simplex-dim1or2}Let $\mathcal{P}$ be a
regular $n$-polytope, $n \geq 2$, that is
\emph{not} a simplex of dimension $n\geq3$. If $\mathcal{P}$ has an odd number of edges $E$, then $n=2$.
\end{lem}
\begin{proof}
From Table \ref{tab:j-Face-Cardinality-and-Shape}, we see that all the regular
3- and 4-polytopes have even $E$. Likewise, substituting $j=1$ into the
$j$-face cardinality formulas for the $n$-crosspolytopes and $n$-cubes in Table \ref{tab:j-Face-Cardinality-and-Shape} yields $E=2^{2}\binom{n}{2}=2n(n-1)$ and $E=2^{n-1}\binom{n}{1}=2^{n-1}n$, respectively, both of which are even. Since $\mathcal{P}$ is not a simplex of dimension $n \geq 3$ by assumption, this means $\mathcal{P}$ is
not a polytope of dimension $n\geq3$. Therefore, $n=2$.
\end{proof}
\subsection{\label{subsec:Generalizing-Fact-2}Generalizing Fact 2}
\begin{thm}
\label{Thm:sum-of-distinct-chords}Let $\mathcal{P}$ be a regular
$n$-polytope inscribed in a unit $n$-sphere ($n \geq 2$).
Let $\mathcal{P}$ have $V$ vertices, $E$ edges, and $k$ chords of distinct lengths, and let $d_{i}$ denote the $i$\textsuperscript{th} distinct chord length
of $\mathcal{P}$.

\begin{enumerate}
\item If $\mathcal{P}$ is a simplex of dimension $n\geq3$, then $\sum_{i=1}^{k}d_{i}^{2}$
is a non-integral rational number.
\item If $\mathcal{P}$ is \emph{not }a simplex of dimension $n\geq3$,
then $\sum_{i=1}^{k}d_{i}^{2}$ is an integer. Specifically:

\begin{enumerate}
\item If $E$ is odd, this integer is $2k+1$.
\item If $E$ is even, this integer is $2k+2$.
\end{enumerate}
\end{enumerate}
\end{thm}
\begin{proof} We prove Part 1, Part 2a, and then Part 2b.

\emph{Part 1.} Let $\mathcal{P}$ be a regular $n$-simplex inscribed
in a unit $n$-sphere where $n\geq3$. Recall that a simplex has only
one type of chord (i.e., its edges), all of which have length $e=2\left(2-\frac{2}{n+1}\right)^{-1/2}$.
Thus, we have
\[
\sum_{i=1}^{k}d_{i}^{2}=e^{2}=\frac{4}{2-\frac{2}{n+1}}=\frac{2\left(n+1\right)}{n}.
\]
Observe that both the numerator and denominator of $\frac{2\left(n+1\right)}{n}$
are integers (and that $n\neq0$). Thus, $\frac{2\left(n+1\right)}{n}$
is a rational number. Moreover, since $n\geq3$, $n$
does not divide 2. On the other hand, since $n$ and $n+1$ are consecutive
integers, they are relatively prime. Hence, since $n\neq1$, $n$
does not divide $n+1$. We conclude that $\frac{2\left(n+1\right)}{n}$
is not an integer but is a rational number.

\emph{Part 2a.} Let $\mathcal{P}$ be an \emph{odd}-edged regular $n$-polytope (inscribed in a unit $n$-sphere) of dimension $n \geq 2$ that is \emph{not} a simplex of dimension $n\geq3$.  Then, by Lemma
\ref{lem:odd-edged-non-simplex-dim1or2}, $\mathcal{P}$ is 2-dimensional.
Hence, by Fact 2, we have $\sum_{i=1}^{k}d_{i}^{2}=E$.
Since $E$ is odd, $\sum_{i=1}^{k}d_{i}^{2}=E=2k+1$ by Lemma
\ref{Lemma-=0000233}.

\emph{Part 2b.} Let $\mathcal{P}$ be an \emph{even}-edged regular $n$-polytope (inscribed in a unit $n$-sphere) of dimension $n \geq 2$ that is \emph{not} a simplex of dimension $n\geq3$. Thus, $\mathcal{P}$
is neither an odd-edged polygon nor a simplex of dimension $n \geq 2$.
Hence, $\mathcal{P}$ is centrally symmetric.

Choose a
vertex \textbf{$\mathbf{P}_{\mathbf{0}}$} of $\mathcal{P}$ and let
\textbf{$\mathbf{P_{i}}$} for $i=1,2,...,k$ be $k$ other vertices
of the polytope such that $\mathbf{P}_{0}$ and $\mathbf{P_{i}}$
form the endpoints of $k$ different chords each having a distinct
length $d_{i}$ and $d_{1}<d_{2}<\cdots<d_{k}$ (just as in Case 3
of the proof of Theorem \ref{Thm:sum-of-all-chords-Vsquared}). From the proof of Theorem \ref{Thm:sum-of-all-chords-Vsquared},
we know that $d_{i}^{2}=2-2\cos\theta_{i}$ for $i=1,2,...,k-1$
(where $\theta_{i}$ is the measure of the angle $\angle\mathbf{P_{0}OP_{i}}$) and $d_{k}^{2}=2^{2}.$  Therefore, we have
\[
\sum_{i=1}^{k}d_{i}^{2}=\sum_{i=1}^{k-1}\left(2-2\cos\theta_{i}\right)+d_{k}^{2}=2\left(k-1\right)-2\sum_{i=1}^{k-1}\cos\theta_{i}+4.
\]
Next, we show that $\sum_{i=1}^{k-1}\cos\theta_{i}=0$. Recall from
the proof of Theorem \ref{Thm:sum-of-all-chords-Vsquared} that, for
each angle of measure $\theta_{i}$ contained in the set $\left\{ \angle\mathbf{P_{0}OP_{i}}:\:i=1,2,...,k\right\} $,
there is an angle (not necessarily distinct from the angle of measure
$\theta_{i}$) that is supplementary to the angle of measure $\theta_{i}$
and is also contained in the set. Thus, if the angle of measure $\theta_{i}$
and its supplement are distinct angles, it follows that $\cos\theta_{i}+\cos\left(\pi-\theta_{i}\right)=\cos\theta_{i}-\cos\theta_{i}=0$
and these two angles contribute nothing to $\sum_{i=1}^{k-1}\cos\theta_{i}$.
If the angle of measure $\theta_{i}$ and its supplement are actually
the same angle, then $\theta_{i}=\frac{\pi}{2}$ and $\cos\left(\frac{\pi}{2}\right)=0$.
Therefore, we see that $\sum_{i=1}^{k-1}\cos\theta_{i}=0$ and $\sum_{i=1}^{N}d_{i}^{2}=2k+2.$
\end{proof}
\begin{cor}
\label{Cor:dual-to-sum-of-distinct}Let $\mathcal{Q}$ be a regular
$n$-polytope circumscribed about a unit
$n$-sphere ($n \geq 2$) having $E$ edges and $F$ facets.  Let $t_{i}$ for
$i=1,2,...,k$ be the distinct lengths of the line segments whose
endpoints are centers of facets.

\begin{enumerate}
\item If $\mathcal{Q}$ is a simplex of dimension $n\geq3$, then $\sum_{i=1}^{k}t_{i}^{2}$
is a non-integral rational number.
\item If $\mathcal{Q}$ is \emph{not }a simplex of dimension $n\geq3$,
then $\sum_{i=1}^{k}t_{i}^{2}$ is an integer. Specifically:

\begin{enumerate}
\item If $E$ is odd, this integer is $2k+1$.
\item If $E$ is even, this integer is $2k+2$.
\end{enumerate}
\end{enumerate}
\end{cor}
\begin{proof}
This corollary follows by much the same reasoning as Corollary \ref{DualCor:sum-of-all-chords}.
Note that the reciprocal of a regular simplex of dimension
$n\geq3$ is a regular simplex of dimension $n\geq3$, the reciprocal of an odd-edged
regular polygon is an odd-edged regular polygon, and the reciprocal
of a centrally-symmetric regular polytope
is a centrally-symmetric regular polytope (cf.
\cite{Coxeter1973}).
\end{proof}

\subsubsection{Special Cases (the regular 3-polytopes)}
\begin{thm}
\label{Cor:special-case-sum-of-distinct}Let $\mathcal{P}$ be a regular
$3$-polytope inscribed in a unit 3-sphere with $V$ vertices and
$k$ distinct chords. Let $d_{i}$ be the $i$\textsuperscript{th}
distinct chord length of $\mathcal{P}$.

\begin{enumerate}
\item For the self-dual tetrahedron, $\sum_{i=1}^{k}d_{i}^{2}\in\mathbb{Q}$.
\item For the dual pair of octahedron and cube, $\sum_{i=1}^{k}d_{i}^{2}=V$.
\item For the dual pair of icosahedron and dodecahedron, $\sum_{i=1}^{k}d_{i}^{2}=2k+2$.
\end{enumerate}
\end{thm}
\begin{proof}Parts 1 and 3 follow directly from Parts
1 and 2c of Theorem \ref{Thm:sum-of-distinct-chords}, respectively
(since the icosahedron and dodecahedron have an even number of edges).
For Part 2, recall that the regular octahedron and
cube have an even number of edges. Hence, for each of them, Part 2c
of Theorem \ref{Thm:sum-of-distinct-chords} guarantees that $\sum_{i=1}^{k}d_{i}^{2}=2k+2$.
We determine $k$ for the octahedron and cube.

\emph{Octahedron.} The octahedron is just the regular 3-crosspolytope,
so consider a regular $n$-crosspolytope of edge length $e$ inscribed
in a unit $n$-sphere (where $n\geq2$). It can be constructed by
(a) creating a Cartesian cross ($n$ mutually orthogonal lines through
a point $\mathbf{O}$; see \cite{Coxeter1973}), (b) finding the points
on each of these lines that are all a distance $e\sqrt{0.5}$ away
from the point $\mathbf{O}$, and (c) connecting all pairs of these points
that do not lie on the same line of the Cartesian cross by line segments
(cf. the construction in \cite{Sequin}).

We now have a regular $n$-crosspolytope whose vertices
are the $2n$ points found in part ``b''
and whose edges are all the line segments constructed in
part ``c.''  (These edges do have length $e$, as can be checked by applying the Pythagorean Theorem
to a triangle whose vertices consist of any one of the pairs of points
mentioned in part ``c'' and the point $\mathbf{O}$).

Now we determine $k$. From part ``c'' of our construction, it is readily apparent that any vertex
$\mathbf{P}$ of the crosspolytope is connected to all other
vertices by edges except for the vertex $\mathbf{P}'$ lying on the same ``Cartesian cross'' line as $\mathbf{P}$. Thus,
the chords of the crosspolytope come in two types: (a) edges and (b)
inner diagonals $\mathbf{PP'}$ that are portions of the ``Cartesian cross'' lines. Thus, $k=2$ and $\sum_{i=1}^{k}d_{i}^{2}=2k+2=6$,
which is the number of vertices.

\emph{Cube.} The cube is the 3-cube, so consider an $n$-cube.
An $n$-cube\textquoteright s $j$-faces are $j$-cubes
(see Table \ref{tab:j-Face-Cardinality-and-Shape}) and its chords
thus come in $n$ types: (a) edges and (b) $n-1$ types of diagonals
(each of which is the \textquotedblleft longest diagonal\textquotedblright{}
of a $j$-face, where $2\leq j\leq n$, of the $n$-cube). Thus,
$k=n$ for a regular $n$-cube and $k=3$ for the 3-cube.
Therefore, $\sum_{i=1}^{k}d_{i}^{2}=2k+2=8$,
which is the number of vertices.
\end{proof}
\begin{cor}
\label{DualCor-to-special-case-of-sum-of-distinct}Let $\mathcal{Q}$
be a regular 3-polytope circumscribed about a unit 3-sphere
with $F$ facets. Let $t_{i}$, for $i=1,2,...,k$, be the distinct
lengths of the line segments whose endpoints are centers of facets
of $\mathcal{Q}$. Then
\begin{enumerate}
\item For the self-dual tetrahedron, $\sum_{i=1}^{k}t_{i}^{2}\in\mathbf{\mathbb{Q}}$.
\item For the dual pair of octahedron and cube, $\sum_{i=1}^{k}t{}_{i}^{2}=F$.
\item For the dual pair of icosahedron and dodecahedron, $\sum_{i=1}^{k}t_{i}^{2}=2k+2$.
\end{enumerate}
\end{cor}

\subsection{Discussion}

Theorem \ref{Thm:sum-of-distinct-chords} (combined with Lemma \ref{lem:odd-edged-non-simplex-dim1or2}) shows that Fact 2 does,
in fact, apply to \emph{all} regular $n$-polytopes (where $n \geq 2$) \emph{except}
most $n$-simplices. Moreover, this theorem shows that the ``$E$''
in Fact 2 can be re-characterized as $2k+1$ (where $k$ is the number
of distinct chords) and specifies the ``some integer'' in Fact 2
as $2k+2$. Thus, these two integers are remarkably similar.

\section{FACT 3: Product of Squared Chords}

The third 2-polytope fact stated that, for any regular 2-polytope
inscribed in a unit 2-sphere, the product of its squared chord lengths
equals $E^{E}.$ We test generalization of this fact to (a) the regular
$n$-crosspolytopes (where $n \geq 2$), (b) the regular 24-cell, and (c) the other regular
polytopes.

\subsection{Preliminaries}

No additional information is need to test generalization of Fact 3
to the $n$-crosspolytopes. However, in testing generalization to
the 24-cell, we will make use of the following definition, and, in testing generalization to the other regular
polytopes, we will make use the cardinalities in Table \ref{tab:number-of-vertices-incident-to-edges-vs}.
\begin{defn}
A \emph{section} of a regular $n$-polytope $\mathcal{P}$ is a
nonempty intersection of $\mathcal{P}$ and an $(n-1)$-plane
$\mathcal{H}$ (cf. \cite{Coxeter1973}). This intersection is a $j$-polytope
for some $j \in \left \{0, 1, ...,  n-1 \right \}$ (cf. \cite{Coxeter1973}).
\end{defn}
\noindent Of particular interest to us are the sections known as \emph{simplified
sections}. These are the sections (a) that are formed when the $(n-1)$-plane
$\mathcal{H}$ is orthogonal to a line $l$ that passes through the
$n$-polytope's center $\mathbf{O}$ and a given, fixed vertex $\mathbf{P_{0}}$
and (b) that form $j$-polytopes whose vertices are all vertices of
the original $n$-polytope $\mathcal{P}$. For these sections, which
we will denote by $\mathcal{S}_{i}$ for $i=0,1,2,...,k$, the distances
between the fixed vertex $\mathbf{P_{0}}$ and each one of the vertices
in a given section ($\mathcal{S}_{i}$ for some particular $i$) are
always the same length (for a regular polytope; cf. \cite{Coxeter1973}).
Hence, if we let $\mathbf{P_{i}}$ be one of the vertices in the section
$\mathcal{S}_{i}$ of a given regular polytope, the lengths of the
chords $\mathbf{P_{0}P_{i}}$ and $\mathbf{P_{0}P_{j}}$ are distinct
if $i\neq j$.

\begin{table}
\caption{\label{tab:number-of-vertices-incident-to-edges-vs}Number of ``Vertices
Incident to Any Edge'' (\textbf{$\nu$}) and Number of ``Edges Incident
to Any Vertex'' (\textbf{$\varepsilon$}) for the Regular 3-Polytopes \cite{Coxeter1973}.}

\medskip{}

\centering{}%
\begin{tabular}{ccc}
\toprule 
\textbf{3-Polytope} & \multicolumn{1}{c}{\textbf{$\nu$}} & \multicolumn{1}{c}{\textbf{$\varepsilon$}}\tabularnewline
\midrule
\midrule 
tetrahedron & 2 & 3\tabularnewline
\midrule 
octahedron & 2 & 4\tabularnewline
\midrule 
cube & 2 & 3\tabularnewline
\midrule 
icosahedron & 2 & 5\tabularnewline
\midrule 
dodecahedron & 2 & 3\tabularnewline
\bottomrule
\end{tabular}
\end{table}
\subsection{\label{subsec:Generalizing-Fact-3}Generalizing Fact 3 to Crosspolytopes}
\begin{thm}
\label{Thm:product-of-all-chords}Let $\mathcal{P}$ be a regular
$n$-crosspolytope with $V$ vertices and $F$ facets inscribed in
a unit $n$-sphere (where $n \geq 2$). Then 
\[
\prod_{i=1}^{N}c_{i}^{2}=F^{V}
\]
where $N$ is the total number of chords of $\mathcal{P}$ and $c_{i}$
is the length of each $i$\textsuperscript{th} chord.
\end{thm}
\begin{proof}
Recall from the proof of Theorem \ref{Cor:special-case-sum-of-distinct} that the chords of
an $n$-crosspolytope ($n\geq2$) come in two types: (a) edges
and (b) inner diagonals. Let the lengths of the latter be denoted by $d$.

An edge is a 1-face, so the number of edges $E$, given by Table \ref{tab:j-Face-Cardinality-and-Shape},
is
\[
E=2^{j+1}\binom{n}{j+1}=2^{2}\binom{n}{2}=2n(n-1).
\]
The number of inner
diagonals, $N_{d}$, is the total number of chords (i.e., $N$) minus
the number of edges. By Lemma \ref{lem:N=00003DV(V-1)/2}, we have:
\[
N_{d}=\frac{V\left(V-1\right)}{2}-E.
\]
Recall from the proof of Theorem \ref{Cor:special-case-sum-of-distinct}
that an $n$-crosspolytope has $2n$ vertices ($n \geq 2)$. Thus, substituting
the values for $V$ and $E$ yields
\[
N_{d}=\frac{(2n)(2n-1)}{2}-2n(n-1)=n.
\]
Now that we know the number of chords of length $e$ and length $d$,
we want to calculate these lengths. From Table
\ref{tab:Edge-Length}, we see $e=2^{1/2}$.
From the
proof of Theorem \ref{Cor:special-case-sum-of-distinct}, we see $d$ is twice the distance from the polytope's center $\mathbf{O}$ to any one of its vertices and that this
distance is $e\sqrt{0.5}$. Thus, $d=2e\sqrt{0.5}=2$.

Therefore, we have:
\[
\prod_{i=1}^{N}c_{i}^{2}=(e^{2})^{E}(d^{2})^{N_{d}}=2{}^{2n(n-1)}(2)^{2n}=\left(2^{n}\right)^{2n}=(2^{n})^{V}.
\]
From Table \ref{tab:j-Face-Cardinality-and-Shape}, the number of
facets is $F=2^{n}\binom{n}{n}=2^{n}$. We conclude that $\prod_{i=1}^{N}c_{i}^{2}=(2^{n})^{V}$=$F^{V}$.
\end{proof}
\begin{cor}
\label{DualCor:product-of-all}Let $\mathcal{Q}$ be an $n$-cube
with $V$ vertices and $F$ facets that circumscribes a unit
$n$-sphere ($n\geq2$). Then $\prod_{i=1}^{N}s_{i}^{2}=V^{F}$
where $s_{i}$ is the length of each $i$\textsuperscript{th} line
segment whose endpoints are centers of facets of $\mathcal{Q}$ and
$N$ is the total number of these line segments.
\end{cor}
\begin{proof}
The reciprocal of a regular $n$-crosspolytope is an $n$-cube
\cite{Coxeter1973}.
\end{proof}

\subsection{Generalizing Fact 3 to the 24-cell}
\begin{thm}
\label{Thm:24-cell-product-of-all-chords}Let $\mathcal{P}$ be a
regular 24-cell with $E$ edges and $R$ ridges inscribed in a unit
4-sphere. Then 
\[
\prod_{i=1}^{N}c_{i}^{2}=6^{E}=6^{R}
\]
where $c_{i}$ is the length of each $i$\textsuperscript{th} chord
of $\mathcal{P}$ and $N$ is the total number of chords.
\end{thm}
\begin{proof} According to Table \ref{tab:Edge-Length}, the edge length of $\mathcal{P}$  is 1. To find its other distinct chord lengths,
we consider the distance $d_{i}'$ from a given, fixed vertex $\mathbf{P_{0}}$
of an arbitrary 24-cell to some vertex $\mathbf{P_{i}}$ in its simplified
section $\mathcal{S}_{i}$ (where $i=1,2,...,k$). (Note that\textemdash
since these $k$ distances are distinct\textemdash they will actually
be the polytope's distinct chord lengths).

In Table V(i) in \cite{Coxeter1973}, there is a list of the values
of $a_{i}$, which is the distance $d_{i}'$ divided by the edge length
$e$ of an arbitrary 24-cell (so $d_{i}'=a_{i}e$).\footnote{Instead of ``$a_{i}$,'' Coxeter \cite{Coxeter1973} actually uses
``$a$.''} Thus, to obtain the distinct chord lengths $d_{i}$ for $\mathcal{P}$, we multiply each of the values of
$a_{i}$ by the edge length of $\mathcal{P}$
(i.e., by 1). These values are listed in the second column of Table
\ref{tab:24-cell-simplified-sects-1}. (For convenience' sake,
the values of $d_{i}^{2}$ are listed in the third column of Table
\ref{tab:24-cell-simplified-sects-1}).

To find the total number of chords of each distinct length, we will
use the total number of vertices $V_i$ in each simplified section $\mathcal{S}_{i}$
as listed in Table V(i) of \cite{Coxeter1973}. Obviously, the number
of chords $m_{i}$ of a given length $d_{i}$ emanating from $\mathbf{P_{0}}$
is the same as $V_i$
(see the fourth column of Table \ref{tab:24-cell-simplified-sects-1}).
Thus, the total number of chords of length $d_{i}$ is $\frac{m_{i}V}{2}$
by Lemma \ref{Lemma:mV/2}. The total number of vertices $V$ of a 24-cell
is $24$ (see Table \ref{tab:j-Face-Cardinality-and-Shape}). Substituting
the values of $m_{i}$ and $V$ into $\frac{m_{i}V}{2}$ yields the
total number of chords of length $d_{i}$ for $i=1,2,...,k$ (as listed
in the last column of Table \ref{tab:24-cell-simplified-sects-1}).

\begin{longtable}{ccccc}
\caption{\label{tab:24-cell-simplified-sects-1}Length and Number of Distinct
Chords in Regular 24-Cell Inscribed in Unit 4-Sphere (second and fourth columns due to \cite{Coxeter1973}).}
\tabularnewline
\hline
\midrule 
\textbf{ $\mathbf{\mathcal{S}}_{\mathbf{i}}$} & $a_{i}$ ($=d_{i}'=d_{i}$) & $d_{i}^{2}$ & \# of vertices ($m_{i})$ & \# of chords of length $d_{i}$\tabularnewline
\midrule
\endfirsthead
\textbf{$\mathbf{\mathcal{S}_{1}}$} & $1$ & $1$ & 8 & 96\tabularnewline
\midrule 
\textbf{$\mathbf{\mathcal{S}_{2}}$} & $2^{1/2}$ & $2$ & 6 & 72\tabularnewline
\midrule 
\textbf{$\mathbf{\mathcal{S}_{3}}$} & $3^{1/2}$ & $3$ & 8 & 96\tabularnewline
\midrule 
\textbf{$\mathbf{\mathcal{S}_{4}}$} & $2$ & $4$ & 1 & 12\tabularnewline
\bottomrule
\end{longtable}
\noindent Therefore, we have $\prod_{i=1}^{k}c_{i}^{2}=\left(1\right)^{96}\left(2\right)^{72}\left(3\right)^{96}\left(4\right)^{12}=6^{96}$.
Table \ref{tab:j-Face-Cardinality-and-Shape} reveals that, for the
24-cell, the number of edges $E$ and the number of ridges $R$ are
both equal to 96. We conclude: $\prod_{i=1}^{k}c_{i}^{2}=6^{E}=6^{R}$
. \end{proof}

\subsection{Generalizing Fact 3 to Other Regular Polytopes}
\begin{thm}
\label{Thm:3D-product-of-all}Let $\mathcal{P}$ be a regular 3-polytope
with $E$ edges and $V$ vertices inscribed in a unit 3-sphere. Then
\[
\prod_{i=1}^{N}c_{i}^{2}=\frac{\nu^{a}}{\varepsilon^{b}}
\]
where $c_{i}$ is the length of each $i$\textsuperscript{th} chord
of $\mathcal{P}$, $\nu$ is the number of vertices incident to any
edge, $\varepsilon$ is the number of edges incident to any vertex,
and $a,\,b$ are positive integers such that $E$ divides $b$ and

\begin{enumerate}
\item for the self-dual tetrahedron,
\[
a\equiv0\:(\textrm{mod }E)\textrm{ and }a\equiv E\:(\textrm{mod }V)
\]
\item for the dual pair of octahedron and cube,
\[
a\equiv V\:(\textrm{mod }E)\textrm{ and }a\equiv E\:(\textrm{mod }V)\textrm{, and}
\]
\item for the dual pair of icosahedron and dodecahedron, 
\[
a\equiv V\:(\textrm{mod }E)\textrm{ and }a\equiv0\:(\textrm{mod }V).
\]
\end{enumerate}
In addition, the number $\varepsilon$ in the equation above may be
replaced by
\[
\frac{E}{\left(E,V\right)}=\frac{\left[E,V\right]}{V}
\]
where $\left(E,V\right)$ is the greatest common divisor of $E$ and
$V$ and $\left[E,V\right]$ is the least common multiple of $E$
and $V$.
\end{thm}
\begin{proof}We consider each of these five cases in turn: (a) the
tetrahedron, (b) the octahedron, (c) the cube, (d) the icosahedron,
and (e) the dodecahedron.

\emph{Tetrahedron.} Recall from the proof of Theorem \ref{Thm:sum-of-all-chords-Vsquared} that a regular $n$-simplex
inscribed in a unit $n$-sphere has $\frac{\left(n+1\right)n}{2}$
edges of length $2\left(2-\frac{2}{n+1}\right)^{-1/2}$ and no diagonals.
Substituting $n=3$ into these expressions shows that, for a regular
3-simplex (i.e., the tetrahedron) inscribed in a unit 3-sphere, $E=6$ and the edge length $e$ is $\frac{2\sqrt{6}}{3}$.
Therefore, we have
\begin{equation}
\prod_{i=1}^{N}c_{i}^{2}=e^{E}=\left(\left(\frac{2\sqrt{6}}{3}\right)^{2}\right)^{6}=\frac{2^{18}}{3^{6}}.\label{eq:3d-product-of-all-tetra}
\end{equation}
Observe that we have $2$ as the numerator\textquoteright s base and
$3$ as the denominator\textquoteright s base (which are $\nu$ and
$\varepsilon$ for the tetrahedron, respectively), as desired. Next,
we consider the exponents. Recall that, for a tetrahedron, $E=6$
and $V=4$. For the denominator\textquoteright s exponent, observe
that $E$ divides 6. For the numerator\textquoteright s exponent,
observe that (a) $18\equiv6\equiv E\:(\mathrm{mod}\,E)$ and (b) $18\equiv2\equiv6\equiv E\:(\mathrm{mod\,}V)$
(as desired). Finally, the identity $EV=(E,V)\cdot[E,V]$ implies
$\frac{E}{\left(E,V\right)}=\frac{\left[E,V\right]}{V}$. Substituting
into the left-hand side of this equation yields $\frac{6}{\left(6,4\right)}=3$,
which is the denominator's base (as desired).

\emph{Octahedron.} Recall from the proof of Theorem \ref{Thm:product-of-all-chords}
that the chords of a regular $n$-crosspolytope inscribed in a unit
$n$-sphere consist of $2n(n-1)$ edges of length $2^{1/2}$ and $n$
inner diagonals of length $2$. Substituting $n=3$ into these expressions
shows that a regular 3-crosspolytope (i.e., the octahedron) inscribed
in a unit 3-sphere has 12 edges of length $2^{1/2}$ (which we denote
by $e$) and three inner diagonals of length $2$ (which we denote
by $d$). Therefore:
\begin{equation}
\prod_{i=1}^{N}c_{i}^{2}=\left(e^{2}\right)^{12}\left(d^{2}\right)^{3}=\left(2\right)^{12}\left(2^{2}\right)^{3}=2^{18}=\frac{2^{18+24q}}{2^{24q}}=\frac{2^{18+24q}}{4^{12q}}\label{eq:3d-product-of-all-octa}
\end{equation}
where $q$ is some integer.

Observe that we have $2$ as the numerator\textquoteright s base and
$4$ as the denominator\textquoteright s base (which are $\nu$ and
$\varepsilon$ for the octahedron, respectively), as desired. Next,
we consider the exponents. Recall that $E=12$ and $V=6$ for the
octahedron. For the denominator\textquoteright s exponent, observe
that $E$ divides $12q$ (as desired). For the numerator\textquoteright s
exponent, observe that (a) $18+24q\equiv0\equiv12\equiv E\:(\mathrm{mod\,}V)$
and (b) $18+24q\equiv6\equiv V\:(\mathrm{mod\,}E)$ (as desired).
Finally, $\frac{12}{\left(12,6\right)}=2$, which is the denominator's base
(as desired).

\emph{Cube.} First, we determine the distinct chord lengths of a cube inscribed in a unit 3-sphere.
Since a cube has eight vertices, we know there must be seven
chords (not necessarily distinct) emanating from a given vertex $\mathbf{P}$.
From Table \ref{tab:number-of-vertices-incident-to-edges-vs}, we
see that there are three edges emanating from $\mathbf{P}$ (three of the seven chords). Table \ref{tab:Edge-Length} in reveals that the cube's edges have length
$3^{-1/2}2$.

For the next distinct chord length, observe that three of the cube's square faces
meet at $\mathbf{P}$ so that each pair of faces
shares an edge emanating from $\mathbf{P}$
(cf. \cite{Coxeter1973}). Since a square has four vertices, each
of these three faces has one vertex that is \emph{not} joined
to $\mathbf{P}$ by an edge. Thus, each face has one diagonal
emanating from $\mathbf{P}$ and each of these three diagonals forms
an (outer) diagonal of the cube. Using Pythagorean's Theorem on the right triangle
formed by one of these diagonals and two adjacent edges shows that these outer diagonals have length $\sqrt{\left(3^{-1/2}2\right)^{2}+\left(3^{-1/2}2\right)^{2}}=\frac{2\sqrt{6}}{3}$. 

We have now determined the length of six of the seven chords emanating
from $\mathbf{P}$. For the seventh, recall
that a cube is centrally-symmetric and, hence, $\mathbf{P}$ has a vertex $\mathbf{P'}$ diametrically opposite to it. Hence, the chord $\mathbf{PP'}$ is a diameter.  Thus, it has length 2 (and is the seventh chord).

Next, we determine the total number of chords of these distinct lengths.
By Lemma \ref{Lemma:mV/2}, viz., $N_{i}=\frac{m_{i}V}{2}$, the cube
has $\frac{3\left(8\right)}{2}=12$ chords of length $3^{-1/2}2$,
$\frac{3\left(8\right)}{2}=12$ chords of length $\frac{2\sqrt{6}}{3}$,
and $\frac{1\left(8\right)}{2}=4$ chords of length 2.

Therefore, we have
\begin{equation}
\prod_{i=1}^{N}c_{i}^{2}=\left(\left(3^{-1/2}2\right)^{2}\right)^{12}\left(\left(\frac{2\sqrt{6}}{3}\right)^{2}\right)^{12}\left(\left(2\right)^{2}\right)^{4}=\frac{2^{68}}{3^{24}}.\label{eq:product-all-chords-cube}
\end{equation}
Observe that 2 is the numerator\textquoteright s base and
3 is the denominator\textquoteright s base ($\nu$ and $\varepsilon$
for the cube, respectively), as desired. Next, we consider the exponents.
Recall that $E=12$ and $V=8$ for the cube. For the denominator\textquoteright s
exponent, observe that $E$ divides 24 (as desired). For the numerator\textquoteright s exponent,
observe that (a) $68\equiv12\equiv E\:(\mathrm{mod}\,V)$ and
(b) $68\equiv8\equiv V\:(\mathrm{mod}\,E)$ (as desired).
Finally, $\frac{12}{\left(12,8\right)}=3$, which is the denominator's base.

\emph{Icosahedron.} The edge length of a regular icosahedron inscribed in a unit 3-sphere is $5^{-1/4}\tau^{-1/2}2$ (see Table
\ref{tab:Edge-Length}). The coordinates
for a regular icosahedron of edge length 2 are given in \cite{Coxeter1973}.
To find the distinct chord lengths of a regular icosahedron of edge
length $5^{-1/4}\tau^{-1/2}2$, we:
\begin{enumerate}
\item Select the coordinates of an arbitrary vertex (of a regular icosahedron
of edge length 2).
\item Compute the distance between this vertex and the other vertices (of
the regular icosahedron of edge length 2).
\item Divide each one of the $k$ distinct distances found in step 3 by
the edge length $2$.
\item Multiply each result in step 3 by the new edge length $5^{-1/4}\tau^{-1/2}2$.
\end{enumerate}
Comparing with \cite[p. 238, para. 3]{Coxeter1973} shows the validity of steps 3 and 4.

We choose the vertex $\left\langle \begin{array}{ccc}
0, & \tau, & 1\end{array}\right\rangle $ and compute the distance between it and the other vertices. To do this, we make repeated use of three
identities, namely, $\tau^{2}=\tau+1$, $\tau^{-1}=\tau-1$, and $\tau^{-2}=-\tau+2$ (given in \cite{Dunlap}):

$\left\langle \begin{array}{ccc}
0, & \tau, & 1\end{array}\right\rangle -\left\langle \begin{array}{ccc}
0, & \tau, & -1\end{array}\right\rangle =\left\Vert \left\langle \begin{array}{ccc}
0, & 0, & 2\end{array}\right\rangle \right\Vert =2$

$\left\langle \begin{array}{ccc}
0, & \tau, & 1\end{array}\right\rangle -\left\langle \begin{array}{ccc}
0, & -\tau, & 1\end{array}\right\rangle =\left\Vert \left\langle \begin{array}{ccc}
0, & 2\tau, & 0\end{array}\right\rangle \right\Vert =2\tau$

$\left\langle \begin{array}{ccc}
0, & \tau, & 1\end{array}\right\rangle -\left\langle \begin{array}{ccc}
0, & -\tau, & -1\end{array}\right\rangle =\left\Vert \left\langle \begin{array}{ccc}
0, & 2\tau, & 2\end{array}\right\rangle \right\Vert =2\sqrt{\tau+2}$

$\left\langle \begin{array}{ccc}
0, & \tau, & 1\end{array}\right\rangle -\left\langle \begin{array}{ccc}
1, & 0, & \tau\end{array}\right\rangle =\left\Vert \left\langle \begin{array}{ccc}
-1, & \tau, & 1-\tau\end{array}\right\rangle \right\Vert =2$

$\left\langle \begin{array}{ccc}
0, & \tau, & 1\end{array}\right\rangle -\left\langle \begin{array}{ccc}
1, & 0, & -\tau\end{array}\right\rangle =\left\Vert \left\langle \begin{array}{ccc}
-1, & \tau, & 1+\tau\end{array}\right\rangle \right\Vert =2\tau$

$\left\langle \begin{array}{ccc}
0, & \tau, & 1\end{array}\right\rangle -\left\langle \begin{array}{ccc}
-1, & 0, & \tau\end{array}\right\rangle =\left\Vert \left\langle \begin{array}{ccc}
1, & \tau, & 1-\tau\end{array}\right\rangle \right\Vert =2$

$\left\langle \begin{array}{ccc}
0, & \tau, & 1\end{array}\right\rangle -\left\langle \begin{array}{ccc}
-1, & 0, & -\tau\end{array}\right\rangle =\left\Vert \left\langle \begin{array}{ccc}
1, & \tau, & 1+\tau\end{array}\right\rangle \right\Vert =2\tau$

$\left\langle \begin{array}{ccc}
0, & \tau, & 1\end{array}\right\rangle -\left\langle \begin{array}{ccc}
\tau, & 1, & 0\end{array}\right\rangle =\left\Vert \left\langle \begin{array}{ccc}
-\tau, & \tau-1, & 1\end{array}\right\rangle \right\Vert =2$

$\left\langle \begin{array}{ccc}
0, & \tau, & 1\end{array}\right\rangle -\left\langle \begin{array}{ccc}
\tau, & -1, & 0\end{array}\right\rangle =\left\Vert \left\langle \begin{array}{ccc}
-\tau, & \tau+1, & 1\end{array}\right\rangle \right\Vert =2\tau$

$\left\langle \begin{array}{ccc}
0, & \tau, & 1\end{array}\right\rangle -\left\langle \begin{array}{ccc}
-\tau, & 1, & 0\end{array}\right\rangle =\left\Vert \left\langle \begin{array}{ccc}
\tau, & \tau-1, & 1\end{array}\right\rangle \right\Vert =2$

$\left\langle \begin{array}{ccc}
0, & \tau, & 1\end{array}\right\rangle -\left\langle \begin{array}{ccc}
-\tau, & -1, & 0\end{array}\right\rangle =\left\Vert \left\langle \begin{array}{ccc}
\tau, & \tau+1, & 1\end{array}\right\rangle \right\Vert =2\tau$

Thus, there are 3 distinct chord lengths: 2, $2\tau$, and $2\sqrt{\tau+2}$.
Dividing these by 2 yields 1, $\tau$, and $\sqrt{\tau+2}$, and multiplying
them by $5^{-1/4}\tau^{-1/2}2$ yields $5^{-1/4}\tau^{-1/2}2$, $5^{-1/4}\tau^{1/2}2$,
and $5^{-1/4}\tau^{-1/2}2\sqrt{\tau+2}$. It can be shown that the
last of these equals $2$.

From our computations, we see that, emanating from the vertex $\left\langle \begin{array}{ccc}
0, & \tau, & 1\end{array}\right\rangle $, there are 5 chords of length 2, 5 chords of length $2\tau$, and
one chord of length $2\sqrt{\tau+2}$. Thus, for the regular icosahedron
of edge length $5^{-1/4}\tau^{-1/2}2$, there are 5 chords of length
$5^{-1/4}\tau^{-1/2}2$, 5 chords of length $5^{-1/4}\tau^{1/2}2$,
and one chord of length $2$. Since a regular icosahedron has $12$
vertices, Lemma \ref{Lemma:mV/2}, viz., $N_{i}=\frac{m_{i}V}{2}$,
shows that there are a total of 30 chords of length $5^{-1/4}\tau^{-1/2}2$,
30 chords of length $5^{-1/4}\tau^{1/2}2$, and six chords of length
$2$. Therefore, we have
\begin{equation}
\prod_{i=1}^{N}c_{i}^{2}=\left(\left(5^{-1/4}\tau^{-1/2}2\right)^{2}\right)^{30}\left(\left(5^{-1/4}\tau^{1/2}2\right)^{2}\right)^{30}\left(\left(2\right)^{2}\right)^{6}=\frac{2^{132}}{5^{30}}.\label{eq:3D-product-of-all-ico}
\end{equation}
Observe that 2 is the numerator\textquoteright s base and
5 is the denominator\textquoteright s base ($\nu$ and $\varepsilon$
for the icosahedron, respectively). Next, we consider
the exponents. Recall that $E=30$ and $V=12$ for the icosahedron.
For the denominator\textquoteright s exponent, observe that $E$ divides
30 (as desired). For the numerator\textquoteright s exponent, observe
that (a) $132=12(11)=V(11)$ and therefore $132\equiv V\:(\mathrm{mod}\,V)$
and that (b) $132=30(4)+12=E(4)+V$ and therefore $132\equiv V\:(\mathrm{mod}\,E)$
as desired. Finally, $\frac{30}{\left(30,12\right)}=5$, which is
the denominator's base.

\emph{Dodecahedron.} The edge length of a regular dodecahedron inscribed in
a unit 3-sphere is $3^{-1/2}\tau^{-1}2$ (see Table
\ref{tab:Edge-Length}). The coordinates
for a regular dodecahedron of edge length $2\tau^{-1}$ are given
in \cite{Coxeter1973}. To find the distinct chord lengths of a regular
dodecahedron of edge length $3^{-1/2}\tau^{-1}2$, we apply the analogous four steps to the four steps listed in the ``Icosahedron case'' (replacing $5^{-1/4}\tau^{-1/2}2$ with $3^{-1/2}\tau^{-1}2$ and 2 with $2\tau^{-1}$).

We choose the vertex $\left\langle \begin{array}{ccc}
0, & \tau^{-1}, & \tau\end{array}\right\rangle$ (which we will denote by $\mathbf{P_0}$) and compute the distance between this vertex and the other vertices:

$\mathbf{P_0} -\left\langle \begin{array}{ccc}
0, & \tau^{-1}, & -\tau\end{array}\right\rangle =\left\Vert \left\langle \begin{array}{ccc}
0, & 0, & 2\tau\end{array}\right\rangle \right\Vert =2\tau$

$\mathbf{P_0} -\left\langle \begin{array}{ccc}
0, & -\tau^{-1}, & \tau\end{array}\right\rangle =\left\Vert \left\langle \begin{array}{ccc}
0, & 2\tau^{-1}, & 0\end{array}\right\rangle \right\Vert =2\tau^{-1}$

$\mathbf{P_0} -\left\langle \begin{array}{ccc}
0, & -\tau^{-1}, & -\tau\end{array}\right\rangle =\left\Vert \left\langle \begin{array}{ccc}
0, & 2\tau^{-1}, & 2\tau\end{array}\right\rangle \right\Vert =2\sqrt{3}$

$\mathbf{P_0}-\left\langle \begin{array}{ccc}
\tau, & 0, & \tau^{-1}\end{array}\right\rangle =\left\Vert \left\langle \begin{array}{ccc}
-\tau, & \tau^{-1}, & \tau-\tau^{-1}\end{array}\right\rangle \right\Vert =2$

$\mathbf{P_0} -\left\langle \begin{array}{ccc}
\tau, & 0, & -\tau^{-1}\end{array}\right\rangle =\left\Vert \left\langle \begin{array}{ccc}
-\tau, & \tau^{-1}, & \tau+\tau^{-1}\end{array}\right\rangle \right\Vert =2\sqrt{2}$

$\mathbf{P_0}-\left\langle \begin{array}{ccc}
-\tau, & 0, & \tau^{-1}\end{array}\right\rangle =\left\Vert \left\langle \begin{array}{ccc}
\tau, & \tau^{-1}, & \tau-\tau^{-1}\end{array}\right\rangle \right\Vert =$2

$\mathbf{P_0} -\left\langle \begin{array}{ccc}
-\tau, & 0, & -\tau^{-1}\end{array}\right\rangle =\left\Vert \left\langle \begin{array}{ccc}
\tau, & \tau^{-1}, & \tau+\tau^{-1}\end{array}\right\rangle \right\Vert =2\sqrt{2}$

$\mathbf{P_0} -\left\langle \begin{array}{ccc}
\tau^{-1}, & \tau, & 0\end{array}\right\rangle =\left\Vert \left\langle \begin{array}{ccc}
-\tau^{-1}, & \tau^{-1}-\tau, & \tau\end{array}\right\rangle \right\Vert =2$

$\mathbf{P_0}-\left\langle \begin{array}{ccc}
\tau^{-1}, & -\tau, & 0\end{array}\right\rangle =\left\Vert \left\langle \begin{array}{ccc}
-\tau^{-1}, & \tau^{-1}+\tau, & \tau\end{array}\right\rangle \right\Vert =2\sqrt{2}$

$\mathbf{P_0} -\left\langle \begin{array}{ccc}
-\tau^{-1}, & \tau, & 0\end{array}\right\rangle =\left\Vert \left\langle \begin{array}{ccc}
\tau^{-1}, & \tau^{-1}-\tau, & \tau\end{array}\right\rangle \right\Vert =2$

$\mathbf{P_0} -\left\langle \begin{array}{ccc}
-\tau^{-1}, & -\tau, & 0\end{array}\right\rangle =\left\Vert \left\langle \begin{array}{ccc}
\tau^{-1}, & \tau^{-1}+\tau, & \tau\end{array}\right\rangle \right\Vert =2\sqrt{2}$

$\mathbf{P_0} -\left\langle \begin{array}{ccc}
1, & 1, & 1\end{array}\right\rangle =\left\Vert \left\langle \begin{array}{ccc}
-1, & \tau^{-1}-1, & \tau-1\end{array}\right\rangle \right\Vert =2\tau^{-1}$

$\mathbf{P_0} -\left\langle \begin{array}{ccc}
1, & 1, & -1\end{array}\right\rangle =\left\Vert \left\langle \begin{array}{ccc}
-1, & \tau^{-1}-1, & \tau+1\end{array}\right\rangle \right\Vert =2\sqrt{2}$

$\mathbf{P_0} -\left\langle \begin{array}{ccc}
1, & -1, & 1\end{array}\right\rangle =\left\Vert \left\langle \begin{array}{ccc}
-1, & \tau^{-1}+1, & \tau-1\end{array}\right\rangle \right\Vert =2$

$\mathbf{P_0} -\left\langle \begin{array}{ccc}
-1, & 1, & 1\end{array}\right\rangle =\left\Vert \left\langle \begin{array}{ccc}
1, & \tau^{-1}-1, & \tau-1\end{array}\right\rangle \right\Vert =2\tau^{-1}$

$\mathbf{P_0} -\left\langle \begin{array}{ccc}
1, & -1, & -1\end{array}\right\rangle =\left\Vert \left\langle \begin{array}{ccc}
-1, & \tau^{-1}+1, & \tau+1\end{array}\right\rangle \right\Vert =2\tau$

$\mathbf{P_0} -\left\langle \begin{array}{ccc}
-1, & -1, & 1\end{array}\right\rangle =\left\Vert \left\langle \begin{array}{ccc}
1, & \tau^{-1}+1, & \tau-1\end{array}\right\rangle \right\Vert =2$

$\mathbf{P_0} -\left\langle \begin{array}{ccc}
-1, & 1, & -1\end{array}\right\rangle =\left\Vert \left\langle \begin{array}{ccc}
1, & \tau^{-1}-1, & \tau+1\end{array}\right\rangle \right\Vert =2\sqrt{2}$

$\mathbf{P_0} -\left\langle \begin{array}{ccc}
-1, & -1, & -1\end{array}\right\rangle =\left\Vert \left\langle \begin{array}{ccc}
1, & \tau^{-1}+1, & \tau+1\end{array}\right\rangle \right\Vert =2\tau$

Thus, there are 5 distinct chord lengths: $2\tau$, $2\tau^{-1}$,
$2\sqrt{3}$, 2, and $2\sqrt{2}$. Dividing these by $2\tau^{-1}$
yields $\tau^{2}$, $1$, $\tau\sqrt{3}$, $\tau$, and $\tau\sqrt{2}$,
and multiplying them by $3^{-1/2}\tau^{-1}2$ yields $3^{-1/2}\tau2$,
$3^{-1/2}\tau^{-1}2$, $2$, $3^{-1/2}2$, and $3^{-1/2}2\sqrt{2}$.

From our computations, we see that, emanating from the vertex
$\left\langle \begin{array}{ccc}
0, & \tau^{-1}, & \tau\end{array}\right\rangle $, there are three chords of length $2\tau$, three chords of length
$2\tau^{-1}$, one chord of length $2\sqrt{3}$, six chords of length
2, and six chords of length $2\sqrt{2}$. Thus, for the regular dodecahedron
of edge length $3^{-1/2}\tau^{-1}2$, there are three chords of length
$3^{-1/2}\tau2$, three chords of length $3^{-1/2}\tau^{-1}2$, one
chord of length $2$, six chords of length $3^{-1/2}2$, and six chords
of length $3^{-1/2}2\sqrt{2}$. Since a regular dodecahedron has $20$
vertices, Lemma \ref{Lemma:mV/2}, viz., $N_{i}=\frac{m_{i}V}{2}$,
shows that there are a total of 30 chords of length $3^{-1/2}\tau2$,
30 chords of length $3^{-1/2}\tau^{-1}2$, 10 chord of length $2$,
60 chords of length $3^{-1/2}2$, and 60 chords of length $3^{-1/2}2\sqrt{2}$.
Therefore, we have $\prod_{i=1}^{N}c_{i}^{2}=$
\begin{equation}
\left(\left(\frac{\tau2}{\sqrt{3}}\right)^{2}\right)^{30}\left(\left(\frac{2}{\tau\sqrt{3}}\right)^{2}\right)^{30}\left(2^{2}\right)^{10}\left(\left(\frac{2}{\sqrt{3}}\right)^{2}\right)^{60}\left(\left(2\sqrt{\frac{2}{3}}\right)^{2}\right)^{60}=\frac{2^{440}}{3^{180}}.\label{eq:product-of-all-dodeca}
\end{equation}
Observe that 2 is the numerator\textquoteright s base and
3 is the denominator\textquoteright s base ($\nu$ and $\varepsilon$
for the dodecahedron, respectively). Next, we consider
the exponents. Recall that $E=30$ and $V=20$ for the dodecahedron.
For the denominator\textquoteright s exponent, observe that $E$ divides 180 as desired. For the numerator\textquoteright s
exponent, observe that (a) $440=20(22)=V(22)$ and therefore $440\equiv V\:(\mathrm{mod}\,V)$
and that (b) $440=30(14)+20=E(14)+V$ and therefore $440\equiv V\:(\mathrm{mod}\,E)$
as desired. Finally, $\frac{30}{\left(30,20\right)}=3$, which is
the denominator's base.
\end{proof}

\subsection{Discussion}

Theorem \ref{Thm:product-of-all-chords} shows that Fact 3 does, in
fact, apply to at least one family of regular $n$-polytopes (where $n \geq 2$)\textemdash
i.e., the family of regular $n$-crosspolytopes\textemdash
\emph{if the ``number of edges'' ($E$) in the base is replaced
by ``number of facets'' ($F$) and the ``number of edges'' ($E$)
in the exponent is replaced by ``number of vertices'' ($V$).} These
replacements are valid since, for 2-polytopes, the facets are, in
fact, the edges and $E=V$.

Theorem \ref{Thm:24-cell-product-of-all-chords} shows that Fact 3
does \emph{not }generalize to the 24-cell but does show that, for
the 24-cell (inscribed in a unit 4-sphere), the product of the squared
chord lengths is a ``very nice'' integer ($6^{E}$, which is equal
to $6^{R}$).

Finally, our proof of Theorem \ref{Thm:3D-product-of-all} shows that
Fact 3 does \emph{not} generalize to the regular 3-simplex, 3-cube,
icosahedron, or dodecahedron and therefore \emph{cannot} generalize
to all regular $n$-simplices or to all $n$-cubes (where $n \geq 2$)
and \emph{probably} does not generalize to the remaining
regular $n$-simplices or $n$-cubes (i.e.,
those of dimension $n>3$) or to the 600-cell (i.e., the ``hypericosahedron'')
or the 120-cell (i.e., the ``hyperdodecahedron'').

\section{FACT 4: Product of Squared Distinct Chords}

The fourth 2-polytope fact stated that, for any regular 2-polytope
inscribed in a unit 2-sphere, the product of its squared distinct
chord lengths equals (a) $E$ (when $E$ is odd) and (b) some integer
(when $E$ is even). We test generalization of this fact to (a) the
regular $n$-crosspolytopes ($n \geq 2$), (b) the 24-cell, and (c) the
other regular polytopes. We need no additional information.

\subsection{Generalizing Fact 4 to Crosspolytopes}
\begin{thm}
\label{Thm:crossPT-product-of-distinct}
Let $\mathcal{P}$ be a regular
$n$-crosspolytope inscribed in a unit $n$-sphere ($n \geq 2$) with $k$ distinct chords. Let $d_{i}$ be
the $i$\textsuperscript{th} distinct chord length.
Then $\prod_{i=1}^{k}d_{i}^{2} = 8$.
\end{thm}
\begin{proof}
Recall from the proof of Theorem
\ref{Thm:product-of-all-chords}
that the chords of a regular $n$-crosspolytope inscribed in a unit
$n$-sphere ($n\geq2$) consist of edges of length $2^{1/2}$
and inner diagonals of length $2$. Hence, $\prod_{i=1}^{N}d_{i}^{2}=\left(2^{1/2}\right)^{2}(2)^{2}=8$.
\end{proof}
\begin{cor}
\label{Cor:crossPTs-product-of-distinct}Let $\mathcal{Q}$ be a regular
$n$-cube circumscribed about a unit $n$-sphere (where $n \geq 2$).
Let $t_{i}$, for $i=1,2,...,k$, be the distinct lengths of the line
segments whose endpoints are centers of facets. Then $\prod_{i=1}^{k}t_{i}^{2} = 8$.
\end{cor}
\subsection{Generalizing Fact 4 to the 24-cell}
\begin{thm}
\label{Thm:24-cell-product-of-distinct}Let $\mathcal{P}$ be a 24-cell
that is inscribed in a unit $4$-sphere and has $V$ vertices, $F$
facets, and $k$ distinct chord lengths. Let $d_{i}$ denote the $i$\textsuperscript{th}
distinct chord length of $\mathcal{P}$. Then 
\[\prod_{i=1}^{k}d_{i}^{2}=F=V.\]
\end{thm}
\begin{proof}
Recall from the proof of Theorem \ref{Thm:24-cell-product-of-all-chords}
that a regular $24$-cell inscribed in a unit $4$-sphere has four
distinct chord lengths: 1, $2^{1/2}$, $3^{1/2},$ and 2.  Squaring these yields 1, 2, 3, and 4.  Thus, $\prod_{i=1}^{k}d_{i}^{2}=4!=24$.
Table \ref{tab:j-Face-Cardinality-and-Shape} reveals that the 24-cell
has 24 vertices and 24 facets. Therefore: $\prod_{i=1}^{k}d_{i}^{2}=F=V$.
\end{proof}
\begin{cor}
\label{DualCor:product-of-distinct}Let $\mathcal{Q}$ be a 24-cell
\emph{circumscribed} about a unit $4$-sphere with $V$ vertices and
$F$ facets. Let $t_{i}$, for $i=1,2,...,k$, be the distinct lengths
of the line segments whose endpoints are centers of facets. Then $\prod_{i=1}^{k}t{}_{i}^{2}=V=F$. 
\end{cor}
\begin{proof}
Note that a regular 24-cell is self-reciprocal \cite{Coxeter1973}.
\end{proof}

\subsection{Generalizing Fact 4 to Other Regular Polytopes}
\begin{thm}
\label{3DThm:product-of-distinct}Let $\mathcal{P}$ be a regular
3-polytope with $E$ edges and $V$ vertices inscribed in a unit 3-sphere,
and let $a$ and $b$ be the exponents in Theorem \ref{Thm:3D-product-of-all}.
Then
\[
\prod_{i=1}^{k}d_{i}^{2}=\frac{\nu^{c}}{\varepsilon^{d}}
\]
where $d_{i}$ is the length of each $i$\textsuperscript{th} distinct
chord of $\mathcal{P}$, $\nu$ is the number of vertices incident
to any edge, $\varepsilon$ is the number of edges incident to any
vertex, and $c,d\in\mathbb{Z}$ such that

\begin{enumerate}
\item for the self-dual tetrahedron, $b=dE$, 
\item for the cube and icosahedron, $b=dE$, and
\item for the octahedron and dodecahedron, $a=cVm$ (where $m=1$ or $2$,
respectively).
\end{enumerate}
In addition, the number $\varepsilon$ in the equation above may be
replaced by either
\[
\frac{E}{\left(E,V\right)}=\frac{\left[E,V\right]}{V}.
\]
\end{thm}
\begin{proof} We consider each of these five cases in turn: (a) the
tetrahedron, (b) the cube, (c) the icosahedron, (d) the octahedron,
and (e) the dodecahedron.

\emph{Tetrahedron.} Recall that a regular tetrahedron inscribed in
a unit 3-sphere has one type of chord: edges of length $\frac{2\sqrt{6}}{3}$.
Therefore, we have 
\begin{equation}
\prod_{i=1}^{k}d_{i}^{2}=\left(\frac{2\sqrt{6}}{3}\right)^{2}=\frac{2^{3}}{3}.\label{eq:product-of-distinct-tetra}
\end{equation}
Observe that 2 is the numerator\textquoteright s base and
3 is the denominator\textquoteright s base ($\nu$ and $\varepsilon$, respectively), as desired. For the denominator\textquoteright s
exponent, recall from the proof of Theorem \ref{Thm:3D-product-of-all}
that $E=6$ and $b=6$ for the tetrahedron. Observe that $d=1=\frac{6}{6}=\frac{b}{E}$,
so $b=dE$ (as desired). Finally, recall from the proof of Theorem
\ref{Thm:3D-product-of-all} that, for the tetrahedron, $\frac{E}{\left(E,V\right)}=\frac{\left[E,V\right]}{V}=3$,
which is the denominator's base.

\emph{Cube.} Recall from the proof of Theorem \ref{Thm:3D-product-of-all}
that a cube inscribed in a unit 3-sphere has three types of chords:
edges of length $3^{-1/2}2$, outer diagonals of length $\frac{2\sqrt{6}}{3}$,
and inner diagonals of length 2. Therefore, we have
\begin{equation}
\prod_{i=1}^{k}d_{i}^{2}=\left(3^{-1/2}2\right)^{2}\left(\frac{2\sqrt{6}}{3}\right)^{2}(2)^{2}=\frac{2^{7}}{3^{2}}.\label{eq:product-of-distinct-cube}
\end{equation}
Observe that the numerator\textquoteright s base is 2 and
the denominator\textquoteright s base is 3 ($\nu$ and $\varepsilon$, respectively). For the denominator\textquoteright s
exponent, recall from the proof of Theorem \ref{Thm:3D-product-of-all}
that $E=12$ and $b=24$ for the cube. Observe that $d=2=\frac{24}{12}=\frac{b}{E}$, so $b=dE$. Finally, recall from the proof of Theorem
\ref{Thm:3D-product-of-all} that, for the cube, $\frac{E}{\left(E,V\right)}=\frac{\left[E,V\right]}{V}=3$,
which is the denominator's base.

\emph{Icosahedron.} Recall that a regular icosahedron inscribed in
a unit 3-sphere has three distinct chord lengths: $\frac{2}{5^{1/4}\tau^{1/2}}$,
$\frac{2\tau^{1/2}}{5^{1/4}}$, and 2. Therefore:
\begin{equation}
\prod_{i=1}^{k}d_{i}^{2}=\left(\frac{2}{5^{1/4}\tau^{1/2}}\right)^{2}\left(\frac{2\tau^{1/2}}{5^{1/4}}\right)^{2}(2)^{2}=\frac{2^{6}}{5}.\label{eq:product-of-distinct-ico}
\end{equation}
Observe that the numerator\textquoteright s base is 2 and
the denominator\textquoteright s base is 5 ($\nu$ and $\varepsilon$, respectively). For the denominator\textquoteright s
exponent, recall from the proof of Theorem \ref{Thm:3D-product-of-all}
that $E=30$ and $b=30$ for the icosahedron. Observe that $d=1=\frac{30}{30}=\frac{b}{E}$, so $b=dE$ (as desired). Finally, recall from the proof
of Theorem \ref{Thm:3D-product-of-all} that, for the icosahedron,
$\frac{E}{\left(E,V\right)}=\frac{\left[E,V\right]}{V}=5$, which
is the denominator's base.

\emph{Octahedron.} Recall the proof of Theorem \ref{Thm:3D-product-of-all}
that a regular octahedron inscribed in a unit 3-sphere has two types
of chords: edges of length $\sqrt{2}$ and inner diagonals of length
2. Therefore, we have
\begin{equation}
\prod_{i=1}^{k}d_{i}^{2}=\left(\sqrt{2}\right)^{2}(2)^{2}=\frac{2^{3+4q}}{2^{4q}}=\frac{2^{3+4q}}{4^{2q}}.\label{eq:product-of-distinct-octa}
\end{equation}
where $q$ is the same as in the proof of Theorem \ref{Thm:3D-product-of-all}.
Observe that the numerator\textquoteright s base is 2 and the denominator\textquoteright s base is 4 ($\nu$ and $\varepsilon$, respectively). For the numerator\textquoteright s
exponent, recall from the proof of Theorem \ref{Thm:3D-product-of-all}
that $V=6$ and $a=18+24q$ for the octahedron. Observe that
$c=3+4q=\frac{18+24q}{6}=\frac{a}{V}$, so $a=cV\cdot1=cVm$.
Finally, recall from the proof of Theorem \ref{Thm:3D-product-of-all}
that, for the octahedron, $\frac{E}{\left(E,V\right)}=\frac{\left[E,V\right]}{V}=2$,
which is the base of the denominator in $\frac{2^{3+4q}}{2^{4q}}$.

\emph{Dodecahedron.} Recall that a regular dodecahedron inscribed
in a unit 3-sphere has five distinct chord lengths: $\frac{2\tau}{\sqrt{3}}$,
$\frac{2}{\sqrt{3}\tau}$, 2, $\frac{2}{\sqrt{3}}$, and $\frac{2\sqrt{2}}{\sqrt{3}}$.
Therefore:
\begin{equation}
\prod_{i=1}^{k}d_{i}^{2}=\left(\frac{2\tau}{\sqrt{3}}\right)^{2}\left(\frac{2}{\sqrt{3}\tau}\right)^{2}\left(2\right)^{2}\left(\frac{2}{\sqrt{3}}\right)^{2}\left(\frac{2\sqrt{2}}{\sqrt{3}}\right)^{2}=\frac{2^{11}}{3^{4}}.\label{eq:product-of-distinct-dodeca}
\end{equation}
Observe that the numerator\textquoteright s base is 2 and the denominator's base is 3 ($\nu$ and $\varepsilon$, respectively). For the numerator\textquoteright s
exponent, recall from the proof of Theorem \ref{Thm:3D-product-of-all}
that $V=20$ and $a=440$ for the dodecahedron. Observe that
$c=11=\frac{440}{40}=\frac{a}{2V}$, so $a=cV\cdot2=cVm$.
Finally, recall from the proof of Theorem \ref{Thm:3D-product-of-all}
that, for the dodecahedron, $\frac{E}{\left(E,V\right)}=\frac{\left[E,V\right]}{V}=3$,
which is the denominator's base.
\end{proof}

\subsection{Discussion}

Theorem \ref{Thm:crossPT-product-of-distinct} shows that Fact 4 does,
in fact, apply to the regular $n$-crosspolytopes (where $n\geq2$) since they have \emph{even}
$E$ and, thus, fall under
the ``even $E$'' case of Fact 4.

Theorem \ref{Thm:24-cell-product-of-distinct} shows that Fact 4 also
applies to the 24-cell (which has even $E$) and shows, for the 24-cell, that the ``some integer'' is the ``number
of facets'' ($F$) or, equivalently, the ``number of vertices''
($V$).

Finally, our proof of Theorem \ref{3DThm:product-of-distinct} shows
that Fact 4 does \emph{not} generalize to the regular 3-simplex, 3-cube,
icosahedron, or dodecahedron and therefore \emph{cannot} generalize
to all regular $n$-simplices or all $n$-cubes ($n \geq 2$)
and \emph{probably} does not generalize to the remaining
regular $n$-simplices or $n$-cubes (i.e.,
those of dimension $n>3$) or to the 600-cell (i.e., the ``hypericosahedron'')
or the 120-cell (i.e., the ``hyperdodecahedron'').
\section{Conclusion}
We have succeeded in finding a single rule
that applies to the chord lengths of \emph{all} regular $n$-polytopes, $n \geq 2$
(namely, the rule given by Theorem \ref{Thm:sum-of-all-chords-Vsquared}). Facts
2-4 generalized to several types of regular polytopes. Given
that the literature contained little information regarding the chord
lengths of regular $n$-polytopes, $n\geq3$, it is clear that
this investigation has taken the study of chord regular
polytopes' chord lengths to a new level (and higher dimensions).

\subsection*{Acknowledgment} The author thanks Dr. Kenneth Wantz of Regent University
for his thorough critiques.

\end{document}